\documentclass[a4paper]{amsart}

\usepackage{amsmath,amssymb,amsthm,a4wide}

\usepackage{tikz}
\usetikzlibrary{calc,intersections,through,backgrounds}

\usepackage{hyperref,caption}
\usepackage{graphicx}
\usepackage{graphics}

\newtheorem{theorem}{Theorem}

\newtheorem*{proposition}{Proposition}

\begin{document}

\title[]{Localization of quantum states \\ and landscape functions}
\author[]{Stefan Steinerberger}
\address{Stefan Steinerberger, Department of Mathematics, Yale University, 10 Hillhouse Avenue, New Haven, CT 06511, USA} \email{stefan.steinerberger@yale.edu}

\begin{abstract} Eigenfunctions in inhomogeneous media can have strong localization 
properties. Filoche \& Mayboroda showed that the function $u$ solving $(-\Delta + V)u = 1$ controls the behavior of eigenfunctions $(-\Delta + V)\phi = \lambda\phi$ via the inequality
$$|\phi(x)| \leq \lambda u(x) \|\phi\|_{L^{\infty}}.$$
This inequality has proven to be remarkably effective in predicting localization and recently Arnold, David, Jerison, Mayboroda \& Filoche connected $1/u$ to decay properties 
of eigenfunctions. We aim to clarify properties of the landscape: the main ingredient is a localized variation estimate obtained from writing $\phi(x)$
as an average over Brownian motion $\omega(\cdot)$ in started in $x$
$$\phi(x) =  \mathbb{E}_{x}\left(\phi(\omega(t)) e^{\lambda t-\int_{0}^{t}{V(\omega(z))dz}} \right).$$
This variation estimate will guarantee that $\phi$ has to change at least by a factor of 2 in a small ball, which implicitly creates a landscape whose relationship with $1/u$ we discuss.

\end{abstract}
\maketitle

\section{Introduction}
\subsection{The Landscape function.} It is well known that physical systems comprised of inhomogeneous materials can exhibit peculiar vibration properties: let $\Omega \in \mathbb{R}^n$ be open, bounded and
\begin{align*}
(-\Delta + V)\phi = \lambda \phi \qquad \mbox{in~}\Omega  ~\mbox{with Dirichlet boundary conditions,}
\end{align*}
where $V:\Omega \rightarrow \mathbb{R}_{\geq 0}$ is a real-valued, nonnegative potential. Anderson \cite{anderson} noticed that for some potentials the low-lying eigenfunctions tend to
strongly localize in a subregion of space in a very complicated manner. 
It seemed difficult to get any information about the localization behavior of these first few eigenfunctions without explicitely computing them. \\

In a truly remarkable contribution, Filoche \& Mayboroda \cite{fil} have given a simple but astonishingly effective method to predict the behavior of low-energy eigenfunctions. Their approach is based on the following inequality (originally due to Moler \& Payne \cite{mol}): if we associate to the problem
a \textit{landscape function} $u:\Omega: \mathbb{R} \rightarrow \mathbb{R}_{}$ given as the solution of
\begin{align*}
(-\Delta + V)u = 1 \qquad \mbox{in~}\Omega \subset \mathbb{R}^n~~~~~ \mbox{with Dirichlet boundary conditions,}
\end{align*}
then there is the inequality
$$ |\phi(x)|  \leq \lambda u(x)  \|\phi\|_{L^{\infty}(\Omega)} .$$ 
The regions where $u$
is small will be of particular interest because an eigenfunction $\phi$ can only localize in $\left\{x: \lambda u(x) \geq 1\right\} \subset \Omega$. The landscape function turns out to be more effective than that: it is instructive
to regard the graph of $u(x)$ as a landscape comprised of 'peaks' and 'valleys'; the valleys may then be understood as inducing a partition of the domain. Numerical experiments \cite{fil} suggest that low-lying
eigenfunctions respect that partition and favor localization in one or at most a few elements in that partition. Moreover, these localized eigenfunctions are 'almost' compactly supported in the sense
that in crossing from one element of the partition to another eigenfunctions seem to experience exponential decay when crossing the valley (see \cite{fil}).

\subsection{The effective potential.} Concerning this exponential drop in size of an eigenfunction when crossing valley, this was recently made precise by Arnold, David, Jerison, Mayboroda \& Filoche \cite{arnold} who point out that the inverse of the landscape function $1/u(x)$ acts as an effective potential responsible for the exponential decay of the localized states (the connection being that $u(x)$ is small in valleys, which makes $1/u(x)$ large and large potentials induce large decay). Their approach is based on writing an eigenfunction as $\phi = u \psi$ for some unknown function $\psi$. The equation
$$ \left( - \Delta + V \right) \phi = \lambda \phi$$
then transforms into
$$ \left[ \frac{1}{u^2} \mbox{div}(u^2 \nabla \psi) \right] +  \frac{1}{u} \psi = \lambda \psi.$$
The new dominating potential $W \equiv 1/u$ is now responsible for the underlying dynamics. The next step is to build an Agmon distance
$$ \rho(r_1, r_2) = \min_{\gamma}\left( \int_{\gamma}{\sqrt{(W(r) - \lambda)_{+}} ds} \right),$$
where $\gamma$ ranges over all paths from $r_1$ to $r_2$
and use Agmon's inequality \cite{agmon} to deduce that for eigenfunctions $\phi$ localized in $r_0 \in \Omega$ 
$$ |\phi(r)| \lesssim e^{-\rho(r_0, r)}.$$
This indicates that $W \equiv 1/u$ is playing a distinguished role. The paper \cite{arnold} also gives convincing numerical evidence that $W - \lambda$
seems to predict decay more accurately than the classical quantity $V-\lambda$. This might seem surprising because clearly $V$ determines the behavior of the eigenfunctions.

\subsection{Organization.} The purpose of our paper is to further clarify these observations and the interplay between an eigenfunction doubling its size in a small ball and the landscape 
function; the main tool is an identity following from the Feynman-Kac formula. More precisely, we

\begin{itemize}
\item derive and discuss the relevant identity,
\item use it to prove a variation estimate localized in a small ball,
\item show how w.r.t. decay $V$ is not as important as a suitable mollification of $V$,
\item compare how $1/u(x)$ fits into that framework
\item and discuss some refinements of the landscape function $u(x)$.
\end{itemize}
We always assume
that $\Omega \subset \mathbb{R}^n$ is bounded with a smooth boundary and $V \in C^2(\Omega)$ to be continuous. 
Technically, this excludes 'block potentials' (which are only $L^{\infty}$) but it is clear that the first few eigenfunctions hardly change if a potential is replaced
by a suitably mollification and therefore the assumption is without loss of generality.

\section{Local analysis of the heat flow}
\subsection{The torision function.} The landscape function arising from $V = 0$,  i.e. the solution of
$$ -\Delta v = 1 \qquad \mbox{in~}\Omega  ~\mbox{with Dirichlet boundary conditions,}$$
is a classical object in shape optimization called the \textit{torsion function}. It appears in elasticity theory \cite{bandle}, heat conduction \cite{vandenberg} and geometry \cite{mark}. 
A version of a landscape function with potential already appeared in the context of homogenization in work of Coifman \& Meyer (unpublished, but see the application to parabolic operators by
S. Wu \cite{wu}). The most prominent role of the torsion function in the field of shape optimization (see e.g. \cite{ban}) is that $v(x)$ gives the expected lifetime of Brownian motion started
in $x$
until it hits the boundary. This suggests to interpret the landscape function in that language; the idea of using Brownian motion to analyze decay properties of eigenfunctions is classical
and
was very successfully used in seminal papers by Carmona \cite{carmona}, Carmona \& Simon \cite{simon1}, Carmona, Masters \& Simon 
\cite{simon2}, Simon \cite{simon3} and others; recently, a similar technique was used by the author \cite{stein} to obtain bounds on the size of nodal sets of Laplacian 
eigenfunctions $\left\{x : \phi(x) = 0\right\}$ on compact manifolds.

\subsection{The idea.} The crucial ingredient is a simple equation representing an eigenfunction $\phi(x)$ as a localized average over local Brownian motion paths running for a short time. 
This equation is not new and has been used earlier for very similar purposes, see e.g. Carmona \& Simon \cite{simon1}.
This equation is obtained by looking at the effect of the 
semigroup $e^{t(\Delta - V)}$ on the eigenfunction. Since eigenfunctions diagonalize the semigroup, we have that if 
$$ (-\Delta + V)\phi = \lambda \phi \qquad \mbox{then} \qquad e^{t(\Delta - V)}\phi = e^{-\lambda t} \phi.$$
At the same time, there is another interpretation of the action of the semigroup in terms of Brownian motion. The
 \textit{Feynman-Kac formula} states that for an arbitrary function $f$
$$ e^{t(\Delta - V)}f(x) = \mathbb{E}_{x}\left(f(\omega(t)) e^{-\int_{0}^{t}{V(\omega(z))dz}} \right),$$
where the expectation $\mathbb{E}_x$ is taken with respect to Brownian motion $\omega(\cdot)$ started in $x$, running for time $t$
and destroyed upon impact on the boundary.
Combining these two equations, we get 
$$ \forall t \geq 0 \qquad \phi(x) =  \mathbb{E}_{x}\left(\phi(\omega(t)) e^{\lambda t-\int_{0}^{t}{V(\omega(z))dz}} \right).$$

This equation describes a complicated relationship between $\phi(x)$, $\lambda$ and $V$, however, it is perfectly suited for establishing a variation estimate in a small ball: assuming the eigenfunction $\phi$ to be essentially constant on a small scale allows us to move the eigenfunction $\phi$ out of the
expectation at the cost of a very small error. We sketch a non-rigorous version of the argument.

\begin{center}
\begin{figure}[h!]
\begin{tikzpicture}[scale=1.2]
\coordinate [label=left:$x_0$] (x) at (0,0);
\coordinate [] (y) at (1,1);
\node (D) [name path=D,draw,circle through=(y),label=left:$B$] at (x) {};
\draw (0,0) \foreach \x in {1,...,400}{--++(rand*0.08,rand*0.08)};
\end{tikzpicture}
\caption{Brownian motion started in $x_0$.}
\end{figure}
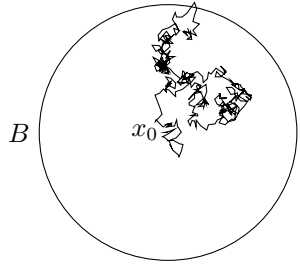
\end{center}

Assume w.l.o.g. that $\phi(x_0) > 0$. Let $B = B(x_0, r)$ be the ball centered at $x_0$ with
maximal radius $r>0$ such that
$$ \forall x \in B: \frac{1}{2} |u(x_0)| \leq |u(x)| \leq 2 |u(x_0)|.$$
Assume furthermore that $V(x_0) \geq \lambda$ and that $V$ is essentially constant on $B$.
We now consider the equation above for $t = c_n r^2$ with $c_n$ a small universal constant depending only on the dimension of $\Omega$. By making $c_n$ sufficiently small,
we can ensure that 99\% of all Brownian paths spend are fully contained in $B$ up to time $t$.  Furthermore,
 $$  \phi(x_0) =  \mathbb{E}_{x_0}\left(\phi(\omega(t)) e^{\lambda t-\int_{0}^{t}{V(\omega(z))dz}} \right) \sim \phi(x_0) \mathbb{E}_{x_0}\left(e^{\lambda t-\int_{0}^{t}{V(\omega(z))dz}} \right).$$
In order for this to hold, we clearly require
$$ \mathbb{E}_{x_0}e^{\lambda t-\int_{0}^{t}{V(\omega(z))dz}} \sim 1 \qquad \mbox{and thus} \qquad -1 \ll \lambda t-\int_{0}^{t}{V(\omega(z))dz}  \leq 0.$$
However, this last quantity can be approximated by
$$ 0 \geq  \lambda t-\int_{0}^{t}{V(\omega(z))dz}  \sim t\left(\lambda - V(x_0)\right) \sim r^2\left(\lambda - V(x_0)\right) \gg -1$$
which is a statement about the maximal size of $r$ depending on $\lambda - V(x_0)$. We will now make this heuristic sketch precise and phrase it in classical terms. However, we emphasize that the most useful way of thinking about the setup to be in terms of path integrals: the condition $V(x) - \lambda \geq c$ could, for example, be replaced by $V(x) - \lambda \geq 0$ and $V(x) - \lambda \geq c$ true 'on average'.

\begin{theorem}[Variation estimate] Suppose $(-\Delta + V)\phi = \lambda \phi$ for $V \geq 0$. There exists a universal constant $c_n$ (depending only on the dimension) such
that the following holds: if, for some $c > 0$,
$$ V(x) - \lambda \geq c \qquad \mbox{or} \qquad V(x) - \lambda \leq -c$$
uniformly on the ball
$$ B = B\left(x_0, \frac{c_n}{\sqrt{c}}  \sqrt{\frac{\lambda}{c} + \log{ \left(c_n  \frac{\|\phi \|_{L^{\infty}}}{| \phi(x_0)|}  \right) }}     \right) \subset \Omega,$$
then we have
$$ \frac{\sup_{x \in B}{|\phi(x)|}}{\inf_{x \in B} |\phi(x)|} \geq 2.$$
\end{theorem}
Opposite statements (especially for the case $V=0$) are usually called 'doubling estimates' and guarantee that an eigenfunction can at most double its size in a certain region of space which 
then bounds the order with which it can vanish around a root (see e.g. Bakri \cite{bakri} or the survey of Zelditch \cite{zeld}). Let us first discuss the case $V(x) - \lambda \geq c$.
Then the result has the expected scaling and translates into
the classical $\sqrt{(V - \lambda)_{+}}$ factor in the Agmon metric. The proof of Theorem 1 yields a stronger result: the uniform estimate $V - \lambda \geq c$
is not necessary (it is only required 'w.r.t. to path integrals') and we can rephrase the condition. 

\begin{quote}
\textit{Local.} $\phi$ locally varies by a constant factor on scale $\sim 1/\sqrt{(V(x)-\lambda)_{+}}$.\\
\textit{Nonlocal.} $\phi$ varies locally by a constant factor on the scale $\sim \sqrt{t_x}$ where 
$$ t_x = \inf_{} \left\{t > 0:   \mathbb{E}_{x}~e^{\lambda t-\int_{0}^{t}{V(\omega(z))dz}} \leq \frac{1}{2} \right\}.$$
\end{quote}
The non-locality of the second
formulation incorporates the diffusive action of the partial differential equation. Moreover, the nonlocal formulation allows for $t$ to be large. 
At the crudest level, the two estimates coincide because
$$  \mathbb{E}_{x}~e^{\lambda t-\int_{0}^{t}{V(\omega(z))dz}}  \sim  1-\left(\lambda - V(x) \right)t.$$
This estimate is only correct up to first order; one might suspect that
the estimate should hold up to second order because Brownian motion is isotropic and therefore
by symmetry
$$ \qquad \mathbb{E}_x \left\langle \nabla V(x), \omega(t) -x\right\rangle = 0.$$
However, there is a curious contribution: Brownian motion moves to distance $\sim \sqrt{t}$ within time $t$ and this has the
curious effect of turning the local geometry of $V$ at second order into a first order
contribution in time
$$ \qquad \mathbb{E}_x  \left\langle \omega(t)-x,  (D^2V)(x) \omega(t)-x \right\rangle = t \Delta V(x).$$
This gives
 $$\mathbb{E}_{x}\ e^{-\int_{0}^{t}{V(\omega(z))dz}}  = 1  - V(x) t + \frac{t^2}{2} \left(V(x)^2 - \Delta V(x)\right) + o(t^2)$$
A geometric interpretation would be as follows: local convexity ($\Delta V(x) > 0$) yields to a stronger decay of the expectation and thus enforces
a stronger decay of the eigenfunction; conversely, local concavity enforces slightly less decay (compared to flat
potentials at the same numerical scale).

\begin{center}
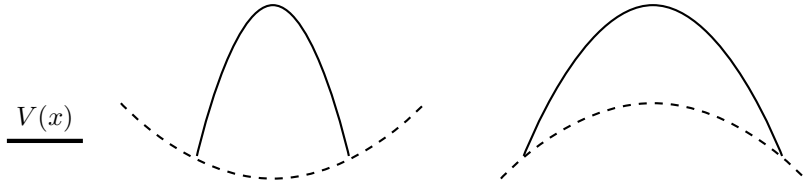
\begin{figure}[h!]
\begin{tikzpicture}[scale=1]
\draw [ultra thick] (-3.5 ,0.5) -- (-2.5,0.5);
\node at (-3,0.8) {$V(x)$};
\draw [dashed, thick, xshift=0cm] plot [smooth, tension=1] coordinates { (-2,1) (0,0) (2,1)};
\draw [thick, xshift=0cm] plot [smooth, tension=1] coordinates { (-1,0.3) (0,2.3) (1,0.3)};

\draw [dashed, thick, xshift=5cm] plot [smooth, tension=1] coordinates { (-2,0) (0,1) (2,0)};
\draw [thick, xshift=5cm] plot [smooth, tension=1] coordinates { (-1.7,0.3) (0,2.3) (1.7,0.3)};
\end{tikzpicture}
\caption{Two potentials having comparable numerical value. Local convexity enforces stronger decay of the eigenfunction, local concavity flatter decay.}
\end{figure}
\end{center}

\subsection{The case $V \leq \lambda$.} The description above only discusses the cases, where locally $V \geq \lambda$ (in order to stress the similarity to Agmon's
inequality), however, it is clear that a similar argument applies whenever $V \leq \lambda$. In that case the expectatation over
the path integrals will grow and this will imply variation at scale $\sqrt{t_x}$ where
$$ t_x = \inf_{} \left\{t > 0:   \mathbb{E}_{x}~e^{\lambda t-\int_{0}^{t}{V(\omega(z))dz}} \geq 2\right\}.$$
Put differently, if $V \leq \lambda$ in such a way that there is local growth in $t$ of $\mathbb{E}_{x}~e^{\lambda t-\int_{0}^{t}{V(\omega(z))dz}}$, then the only
way for
 $$  \phi(x) =  \mathbb{E}_{x}\left(\phi(\omega(t)) e^{\lambda t-\int_{0}^{t}{V(\omega(z))dz}} \right) $$
to be valid is for $\phi(\omega(t))$ to be, on average, smaller than $\phi(x)$
which indicates doubling of size on the length scale $\sqrt{t_x}$. In particular, the nonlocal formulation also applies to the
setting where the classical $\sqrt{(V-\lambda)_{+}}$ yields no more information.
Using this formulation, we can recover the classical intuition for
Laplacian eigenfunctions: if $-\Delta \psi = \kappa \psi$, then one expects $\psi$ to oscillate on the wavelength $\sim \kappa^{-1/2}$. In
our case, if $V \ll \lambda$, then
$$-\Delta \psi = (\lambda - V)\psi \sim \lambda \psi$$
and the variation estimate guarantees oscillation on scale $\sqrt{\lambda}/c$, where
$c \sim \lambda - V \sim \lambda$ and thus $\sqrt{\lambda}/c \sim \lambda^{-1/2}.$ 
Another way of seeing this is that we expect a doubling on the scale $\sim \sqrt{t_x}$ and whenever $V \ll \lambda$, then
$$ t_x = \inf_{} \left\{t > 0:   \mathbb{E}_{x}~e^{\lambda t-\int_{0}^{t}{V(\omega(z))dz}} \geq 2\right\} \sim  \inf_{} \left\{t > 0:   \mathbb{E}_{x}~e^{\lambda t} \geq 2\right\} \sim \frac{1}{\lambda}.$$

\subsection{Landscape from variation.} So far, we have only discussed the variation estimate; the relationship with landscapes 
is easily visualized: pick a sequence of points $x_1, x_2, \dots, x_k \dots \in \Omega$, compute $t_{x_j}$ and draw circles with radius $\sqrt{t_{x_j}}$
around the points. Smaller circles correspond to variation by a factor on a smaller spatial scale (i.e. faster growth/decay). The 'valleys' of the landscape
correspond to regions with smaller circles (crossing them causes a variation by a factor 2 and thus frequent crossing is equivalent to either large growth or large decay) 
whereas 'hills' correspond to regions larger circles. This essentially reproduces the landscape generated by $u$ since both $\sqrt{t_x}$ and $u$ may be regarded as mollifications of $1/V$.

\begin{figure}[h!]
\centering
\includegraphics[width=0.4\textwidth]{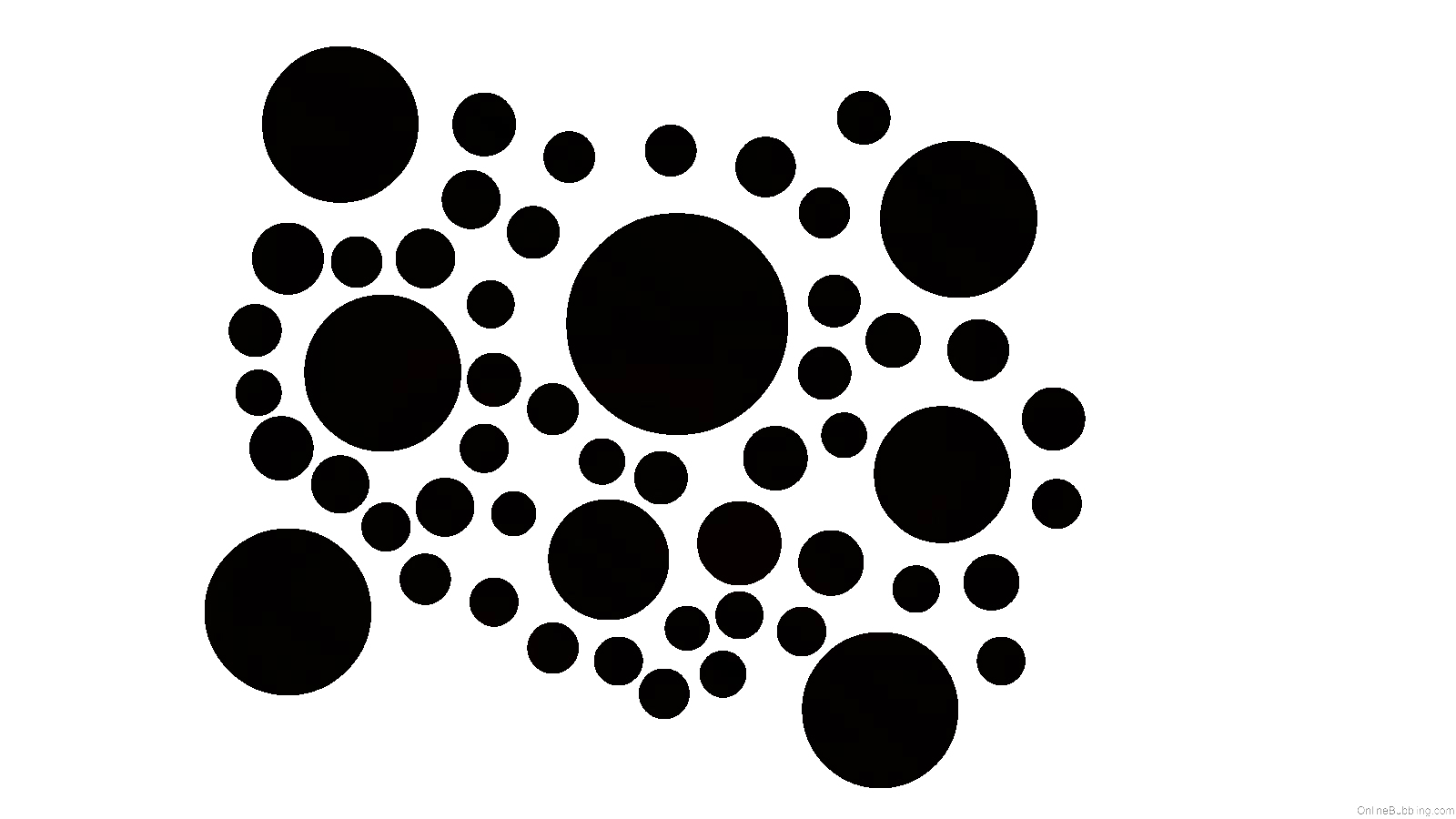}
\caption{An implicit description of the landscape by $\sqrt{t_{x_j}}$ balls around points $x_j$.}
\end{figure}

The computation of $t_{x_j}$ is clearly nontrivial, however, up to first order it is easy, since
$$  \mathbb{E}_{x} e^{-\int_{0}^{t}{V(B(\omega(z)))dz}} \sim  1 - \mathbb{E}_{x} \int_{0}^{t}{V(B(\omega(z)))dz}$$
and, by linearity, we can exchange expectation and integration
$$
 \mathbb{E}_{x} \int_{0}^{t}{V(B(\omega(z)))dz} = \int_{0}^{t}{   \mathbb{E}_{x}  V(B(\omega(z)))dz} \sim \int_{0}^{t}{  (e^{z\Delta} V)(x)dz}  = \left[ \left( \int_{0}^{t}{e^{z \Delta}dz} \right)*V\right](x).
$$
This is essentially accurate as along as $d(x, \partial \Omega) \gg \sqrt{t}$ while it 
overestimates the quantity as soon as $d(x, \partial \Omega) \lesssim \sqrt{t}$. In contrast, the landscape function $\lambda u(x)$ has a simpler and more convenient linear scaling 
in the eigenvalue $\lambda$. Indeed, $\lambda u(x)$ arises naturally from the following heuristic, which nicely clarifies 
how $u(x)$ interacts locally with the semigroup induced by $(-\Delta + V)$. \\

\textit{Heuristic.} We compute an expansion of $e^{t (\Delta - V)} u(x)$ for $t$ small in two different ways. Note that 
$$ (-\Delta + V)u = 1 \qquad \mbox{and thus, for $t$ small,} \qquad e^{t (\Delta - V)} u(x) = u(x) - t + o(t).$$
However, the semigroup may also be expanded using Feynman-Kac and for $t$ small
\begin{align*}
  (e^{t (\Delta - V)} u)(x) &= \mathbb{E}_{x}\left(u(\omega(t)) e^{-\int_{0}^{t}{V(\omega(z))dz}} \right) \\
&= \mathbb{E}_{x}\left(u(\omega(t)) \left(1-\int_{0}^{t}{V(\omega(z))dz}\right) \right)\\
& \sim u(x) -  u(x)\mathbb{E}_{x}\int_{0}^{t}{V(\omega(z))dz}.
\end{align*}
By matching the coefficient in the linear term, we get that
$$ u(x)\mathbb{E}_{x}\int_{0}^{t}{V(\omega(z))dz} \sim t.$$
Therefore, for $t$ sufficiently small
$$ \mathbb{E}_{x}\int_{0}^{t}{V(\omega(z))dz} \sim \frac{t}{u(x)}$$
and thus
$$  \mathbb{E}_{x}e^{\lambda t-\int_{0}^{t}{V(\omega(z))dz}} \sim 1 + \left( \lambda - \frac{1}{u(x)} \right) t.$$ 
This heuristic naturally recovers decay in the region $\left\{x:1/u(x) > \lambda \right\}$. The only inaccuracy is 
that we actually have $ \mathbb{E}_{x} u(\omega(t)) \sim u(x) + t\Delta u(x)$ .
Indeed, unless $|\Delta u(x)|$ is very big, we get from
$$ (-\Delta + V)u(x) = 1 \qquad \mbox{that} \qquad u(x) \sim \frac{1}{V(x)}.$$
In that case it is not very surprising that $1/u(x)$ is effective at predicting decay. However, the relationship runs
deeper than that: using again Feynman-Kac we get
\begin{align*}
e^{t(\Delta - V)}1 &=   \mathbb{E}_{x}e^{-\int_{0}^{t}{V(\omega(z))dz}} \\
e^{t(\Delta - V)^{-1}}1 &= 1 - tu(x) +o(t)
\end{align*}
Furthermore, information up to second order appears with the correct sign. We can rewrite $(-\Delta + V)u=1$ as
$$ u = \frac{1}{V} + \frac{\Delta u}{V}.$$
Assuming $\Delta u \sim 0$ gave the first-order approximation $u \sim 1/V$. The next natural step is
to iterate this 
$$ u \sim \frac{1}{V} + \frac{\Delta \left(\frac{1}{V}\right)}{V}  = \frac{1}{V} + 2\frac{|\nabla V|^2}{V^3} - \frac{\Delta V}{V^3}.$$ 
Restricting to local extrema, we see that (up to lower order terms) $1/u$ is
bigger than $V$ in local minima and smaller than $V$ in local maxima (thus recreating the behavior from above).

\section{New landscape functions}

\subsection{Nonlocal refinement.} 
 From now on, a 'landscape function' refers to any function
$h(x)$ with the property that for a fixed eigenvalue $\lambda$ the eigenfunction $(-\Delta + V)\phi = \lambda \phi$ satisfies
$$ |\phi(x)|   \leq h(x)  \| \phi \|_{L^{\infty}}.$$
$h(x) = 1$ is trivially admissible. We will continue to use $u(x)$ to denote
the classical landscape function given as the solution of $(- \Delta + V)u = 1$. It satisfies the inequality with $h(x) = \lambda u(x)$. 

\begin{theorem}[Landscape bootstrapping] Suppose $(-\Delta + V)\phi = \lambda \phi$ and $h(x)$ satisfies
$$ |\phi(x)| \leq h(x) \| \phi \|_{L^{\infty}},$$
then the same inequality holds for $h(x)$ replaced by
\begin{align*}
 h_1(x) &= \inf_{t \geq 0}{  \mathbb{E}_x  \left(  h((\omega_t) e^{\lambda t -\int_{0}^{t}{V(\omega(s))ds} }  \right)   } \\
&= \inf_{t \geq 0}{ e^{\lambda t} e^{t(\Delta - V)} h(x).   } 
\end{align*}
\end{theorem}
By letting $t \rightarrow 0$, we get $h_1(x) \leq h(x)$.
There is a delicate balance between two terms:
$e^{t(\Delta - V)} h(x)$ is a diffusion semigroup inducing exponential decay which is counteracted by the exponential growth $e^{\lambda t}$. 
 It is interesting to note that the function $\lambda u(x)$ plays a distinguished role
and is characterized by the fact that no purely local (i.e. $t \rightarrow 0$) considerations have any effects.\\

\begin{proposition} Suppose $f \in C^2(\Omega)$ satisfies
$$ \frac{d}{dt} e^{\lambda t} e^{t(\Delta - V)} f(x) \big|_{t=0} = 0 \qquad \mbox{then} \qquad f(x) = \lambda u(x).$$
\end{proposition}
\begin{proof} The argument is immediate. Note that for $t$ small
\begin{align*}
 e^{\lambda t} e^{t(\Delta - V)} f(x)  &= (1 + \lambda t + \mathcal{O}(t^2)))(1 + t (\Delta - V)  f(x)  + \mathcal{O}(t^2))\\
&= 1 +t ( \lambda  +  (\Delta - V)f(x) )+ \mathcal{O}(t^2))
\end{align*}
and therefore
$ f(x) = (-\Delta + V)^{-1}\lambda = \lambda u(x).$
\end{proof}

\subsection{Computational tricks.}
The main results from the previous section allow for the generation of new landscape functions out of $\lambda u(x)$, however,
one should consider that solving $e^{t(\Delta - V)}u(x)$ may be as hard or harder than directly computing eigenfunctions.
The purpose of this section is to suggest a cheap way of creating computationally feasible improvements.
The original proofs \cite{fil, mol} demonstrating
$$|\phi(x)| \leq \lambda u(x) \|\phi\|_{L^{\infty}}$$
use Green's functions. A very simple argument that we could not find in the literature is as follows.
\begin{proof} $\phi$ is an eigenfunction and $-\Delta + V$ is elliptic operator. The maximum principle yields
$$ \phi(x) = (-\Delta + V)^{-1} \lambda \phi(x) \leq  (-\Delta + V)^{-1} \lambda \|\phi\|_{L^{\infty}} = \lambda \|\phi\|_{L^{\infty}}  (-\Delta + V)^{-1}1,$$
where the last term is precisely $u(x)$.
\end{proof}
  The very simple proof immediately suggests two improvements. The first improvement would be to iterate the inequality: let
$k \in \mathbb{N}$ be arbitrary. Then we have the inequality
$$ \phi(x) = (-\Delta + V)^{-k} \lambda^k \phi(x) \leq  (-\Delta + V)^{-k} \lambda^k \|\phi\|_{L^{\infty}} = \lambda^k \|\phi\|_{L^{\infty}}  (-\Delta + V)^{-k}1.$$
This variant was already known to Filoche \& Mayboroda; its downside is that the increased power on the eigenvalue tends to make the bounds
worse for higher eigenfunctions. There exists an elementary improvement that preserves linear scaling in the eigenvalue.

\begin{proposition} If $(-\Delta + V) \phi = \lambda \phi$, then
$$|\phi(x)| \leq \left( \lambda(-\Delta + V)^{-1} \min(\lambda u(x), 1) \right) \|\phi\|_{L^{\infty}}$$
\end{proposition}
\begin{proof} We bootstrap the original inequality and have
\begin{align*}
 \phi= (-\Delta + V)^{-1} \lambda \phi &\leq  \lambda (-\Delta + V)^{-1}  |\phi| \leq \lambda (-\Delta + V)^{-1} \min(\lambda u(x),1) \|\phi\|_{L^{\infty}}
\end{align*}
\end{proof}

\begin{figure}[h!]
\centering
\includegraphics[width=0.6\textwidth]{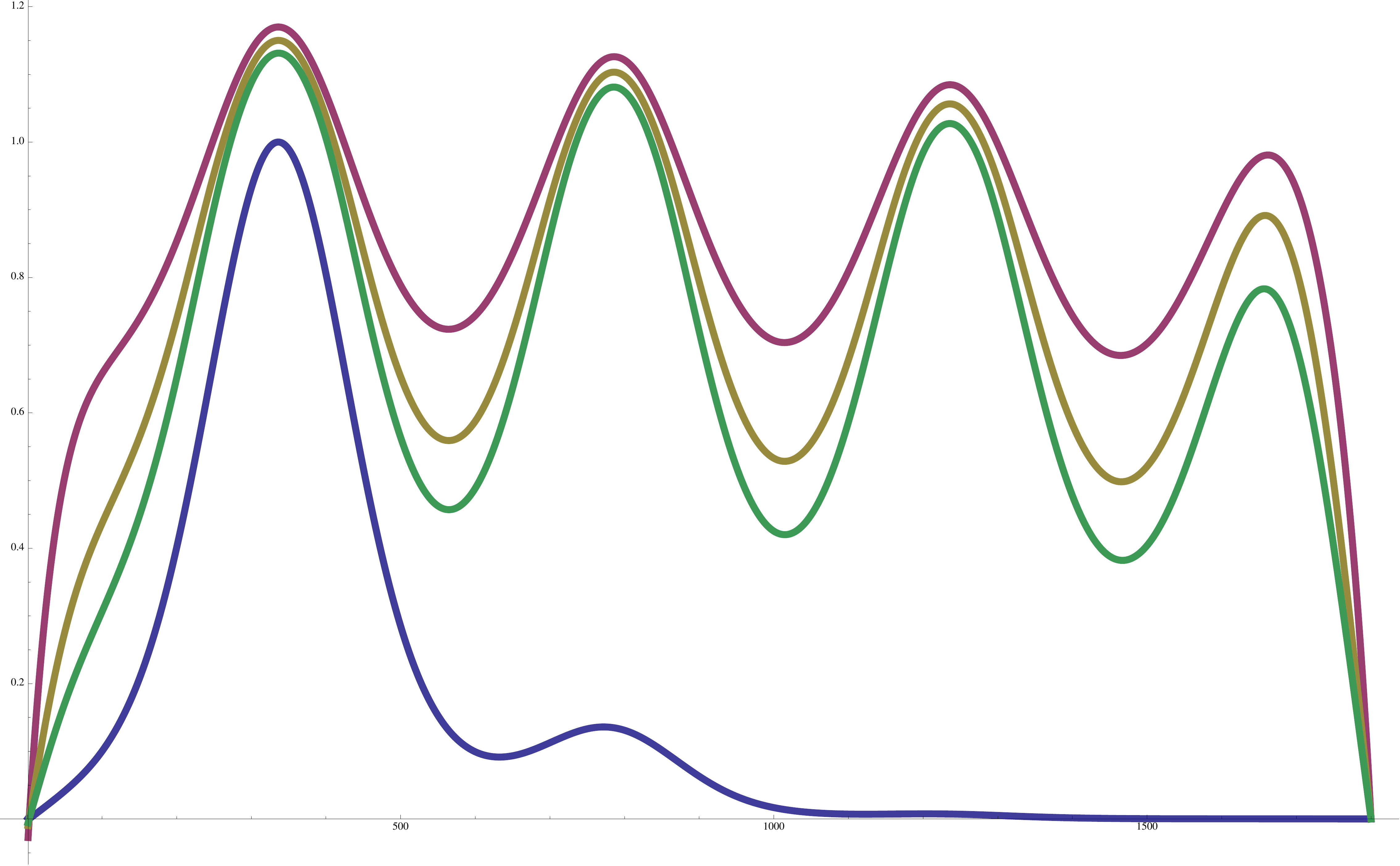}
\captionsetup{width=0.8\textwidth}
\caption{The profile $\phi_1/\|\phi_1\|_{L^{\infty}}$ (blue), the landscape function $u(x)$ (purple) and the first two iterations (yellow, green) of Proposition 3.}
\end{figure}

The reason why such a simple argument could indeed be effective is that the landscape function is sometimes bigger than 1; in that
case we can perform a simple cut-off at 1 and iterate to get additional information; alternatively, one could use the bound from the
previous section in a neighborhood to propagate the gain from the cutoff to nearby regions. Clearly, that simple argument could
also be iterated and, by the same token,
$$|\phi(x)| \leq \left(   \lambda(-\Delta + V)^{-1}  \min \left(1,   \lambda(-\Delta + V)^{-1} \min(\lambda u(x), 1) \right) \right) \|\phi\|_{L^{\infty}}.$$
Like the original landscape function, these improvements can only be effective for $\lambda$ small: for $\lambda$ large, we will have
$ \lambda u(x) \geq 1$ everywhere except for small regions close to the boundary and thus $\min(\lambda u(x), 1) \sim 1$ and the argument will merely recreate
$u$.

\section{Proofs}
\subsection{Proof of Theorem 1.} 
The main idea has already been outlined above. We use
$$ \phi(x_0) =  \mathbb{E}_{x_0}\left(\phi(\omega(t)) e^{\lambda t-\int_{0}^{t}{V(\omega(z))dz}} \right)$$
for a suitable time $t$. Almost all paths are contained within a ball of radius $\sim \sqrt{t}$ (suggesting to
set $t=1/c$). However, some Brownian motion paths will actually leave that ball and may give a large
additional contribution but the likelihood of that happening is small. Using the standard reflection principle implies for one-dimensional Brownian motion $B(t)$ that
$$ \mathbb{P}\left( \max_{0 < s < t}{B(s)} \geq a \right) = 2 \mathbb{P}(B(t) \geq a).$$
This implies that the likelihood of a Brownian motion leaving a ball of fixed radius for some time $0 < s < t$ can be bounded by the
likelihood of having left the ball at time $t$.
More precisely, the standard Gaussian heat kernel estimate yields for some universal constants $c_1, c_2 > 0$ (depending only on the dimension) and fixed
$t > 0$
$$ \mathbb{P}(\|\omega(t) - \omega(0)\| \geq \delta t^{1/2}) \leq c_1 e^{-c_2 \delta^2}.$$
This holds at a much greater level of generality \cite{fern}. The proof consists of
showing that all the quantities work together as described by doing the algebra.

\begin{proof} 
We assume that
$$ \forall x \in B: \frac{1}{2} |\phi(x_0)| \leq |\phi(x)| \leq 2 |\phi(x_0)|$$
with 
$$ B = B\left(x_0, \frac{c_3}{\sqrt{c}}  \sqrt{\frac{\lambda}{c} + \log{ \left(c_3  \frac{\|\phi \|_{L^{\infty}}}{| \phi(x_0)|}  \right) }}     \right) \subset \Omega,$$
where $c_3$ is some constant we are allowed to choose depending on $c_1, c_2$.
We start with assuming that
$$ V - \lambda \geq c \qquad \mbox{uniformly on}~B.$$
The case $V(x) - \lambda \leq -c$ is similar and will be described afterwards.
We can assume without loss of generality that $\phi(x_0) > 0$. The time scale of
the argument will be $ t = 1/c$.
For a Brownian motion $\omega(s)$ started in $x_0$ we distinguish two cases:
\begin{itemize} 
\item \textbf{Case 1 (generic).} $ \left\{ \omega(s): 0 \leq s \leq t \right\} \subset B$
\item \textbf{Case 2 (rare).} $ \left\{ \omega(s): 0 \leq s \leq t \right\} \not\subset B.$
\end{itemize}
Case 1 is very easy to deal with: in that case, we can easily bound any single path fully via
\begin{align*}
\phi(\omega(t)) e^{\lambda t-\int_{0}^{t}{V(\omega(z))dz}}  &\leq 2 \phi(x) e^{ \int_{0}^{t}{( \lambda - V(\omega(z)))dz}}  \\
&\leq2e^{- t c} \phi(x_0) 
\end{align*}
and the same argument holds in expectation for all paths conditioned on Case 1.
 It remains to consider Case 2. If it occurs, then we only have the trivial bound (using $V \geq 0$)
$$ \phi(\omega(t)) e^{\lambda t-\int_{0}^{t}{V(\omega(z))dz}}  \leq \|\phi \|_{L^{\infty}} e^{\lambda t},$$
which is large but the likelihood of that event is small. Combining this with the upper bound on the likelihood of Case 2 gives
\begin{align*}
\phi(x_0) =  \mathbb{E}_{x_0}\left(\phi(\omega(t)) e^{\lambda t-\int_{0}^{t}{V(\omega(z))dz}} \right) \leq 2e^{- tc}\phi(x_0) + c_1 e^{-c_2 \delta^2} e^{\lambda t} \|\phi \|_{L^{\infty}}.
\end{align*}
We derive a contradiction by setting $t = 1/c$, which gives
$$ \delta = c_3\sqrt{\log{ \left(c_3  \frac{e^{\frac{\lambda}{c}} \|\phi \|_{L^{\infty}}}{| \phi(x_0)|}  \right)}  }.$$
Simple algebra allows us to reformulate the inequality as
$$ \phi(x_0) \leq 2 e^{-1} \phi(x_0) + \varepsilon_{c_1, c_2, c_3} \phi(x_0),$$
where
$$ \varepsilon_{c_1, c_2, c_3} = \frac{c_1}{c_3^{c_2 c_3^2}} e^{-\frac{\lambda}{c}(c_2 c_3^2 - 1)} \left( \frac{|\phi(x_0)|}{\|\phi\|_{L^{\infty}}} \right)^{c_2 c_3^2 - 1} $$
 can be made arbitrarily small by making $c_3$ sufficiently large (depending only on $c_1, c_2$). This gives a contradiction since $2e^{-1} < 1$ and concludes the first case. The other statement to consider is
$$ V - \lambda \leq -c \qquad \mbox{uniformly on}~B.$$
Here, essentially all signs are reversed and we show that there is too much local growth.
We can again distinguish Case 1 and Case 2 and get for Case 1 that
 \begin{align*}
\phi(\omega(t)) e^{\lambda t-\int_{0}^{t}{V(\omega(z))dz}}  &\geq \frac{1}{2} \phi(x_0) e^{ \int_{0}^{t}{( \lambda - V(\omega(z)))dz}}  \\
&\geq \frac{1}{2}e^{t c} \phi(x_0).
\end{align*}
The second case is again completely without control and we can only use the trivial estimate
$$ \phi(\omega(t)) e^{\lambda t-\int_{0}^{t}{V(\omega(z))dz}}  \geq -e^{\lambda t}\|\phi\|_{L^{\infty}}.$$
Altogether, this yields
\begin{align*}
\phi(x_0) =  \mathbb{E}_{x_0}\left(\phi(\omega(t)) e^{\lambda t-\int_{0}^{t}{V(\omega(z))dz}} \right) &\geq \mathbb{P}(\mbox{Case 1}) \frac{1}{2}e^{tc}\phi(x_0) - c_1 e^{-c_2 \delta^2} e^{\lambda t} \|\phi \|_{L^{\infty}}.
\end{align*}
For $c_3$ sufficiently large, we can ensure that $\mathbb{P}(\mbox{Case 1}) \geq 1/2$. Thus, for $c_3$ sufficiently large
\begin{align*}
\phi(x_0) =  \mathbb{E}_{x_0}\left(\phi(\omega(t)) e^{\lambda t-\int_{0}^{t}{V(\omega(z))dz}} \right) &\geq \frac{1}{4}e^{tc}\phi(x_0) - c_1 e^{-c_2 \delta^2} e^{\lambda t} \|\phi \|_{L^{\infty}}.
\end{align*}
Plugging things in as before yields (using again $t = 1/c$ and therefore same value of $\delta$ as before)
$$ \phi(x_0) \geq \frac{e^{\delta}}{4}\phi(x_0) - \varepsilon_{c_1, c_2, c_3} \phi(x_0),$$
where by the same computation as above $e^{\delta}/4$ can be made bigger than 2 by choosing $c_3$ sufficiently large and
$$ \varepsilon_{c_1, c_2, c_3} = \frac{c_1}{c_3^{c_2 c_3^2}} e^{-\frac{\lambda}{c}(c_2 c_3^2 - 1)} \left( \frac{|\phi(x_0)|}{\|\phi\|_{L^{\infty}}} \right)^{c_2 c_3^2 - 1} $$
can be made arbitrarily small by making $c_3$ sufficiently large.
\end{proof}

\subsection{Proof of the Theorem 2.}
\begin{proof}
The proof used uses the identity for all $t \geq 0$ to introduce an infimum.
\begin{align*}
\phi(x) = (e^{\lambda t}e^{t (\Delta - V)}\phi)(x) &=  \mathbb{E}_x  \left(  \phi(B_t(\omega)) e^{\lambda t -\int_{0}^{t}{V(B(\omega(s)))ds} }  \right)   \\
&=  \inf_{t \geq 0}{  \mathbb{E}_x  \left(  \phi(B_t(\omega)) e^{\lambda t -\int_{0}^{t}{V(B(\omega(s)))ds} }  \right)}. 
\end{align*}
By assumption $\phi$ is dominated by a landscape function $h(x)$ and thus
\begin{align*}  \inf_{t \geq 0}{  \mathbb{E}_x  \left(  \phi(B_t(\omega)) e^{\lambda t -\int_{0}^{t}{V(B(\omega(s)))ds} }  \right)}  &\leq 
 \|\phi\|_{L^{\infty}} \inf_{t \geq 0}{  \mathbb{E}_x  \left(  h(B_t(\omega)) e^{\lambda t -\int_{0}^{t}{V(B(\omega(s)))ds} }  \right)} \\
&=  \|\phi\|_{L^{\infty}}  \inf_{t \geq 0}{ e^{\lambda t} e^{t(\Delta - V)} h(x)}.
\end{align*}
This concludes the argument.
\end{proof}

\subsection{Refined variation estimate} Since Brownian motion is isotropic, the proof of the variation estimate can cover 
a stronger result: simply put, the variation estimate is true because of \text{curvature} of the graph. More generally, we can show that if
$ h(x): \Omega \rightarrow \mathbb{R}$ is another function satisfying $h(x_0) = 0$.
$$ \forall x \in B \qquad \left| e^{\lambda t}e^{t (\Delta -V)}h(x)\right| \leq \frac{1}{100}\phi(x)$$
for the value of $t$ for which we wish to apply the variation estimate,
then the function
$$ \phi(x) - h(x) \qquad \mbox{doubles its sizes on} ~ B.$$
The best example is perhaps given by $V=0$ (eigenfunctions of the Laplace operator $-\Delta \phi = \lambda \phi$) and $h(x)$ being the best linear approximation of $\phi$ in $x_0$
$$ h(x) = \left\langle \nabla \phi(x_0), x- x_0 \right\rangle.$$
This function is essentially invariant under heat flow for $d(x_0, \partial \Omega) \gg \sqrt{t}$ (with a negligible contribution from $|\nabla \phi(x_0)|$ that can be made precise). The variation estimate implies that $\phi(x) - h(x)$ doubles
its size and since we have removed the tangent plane this implies that $\phi(x) - h(x)$ can't be too small because there is curvature in the graph.
This also explains why the variation estimate does not apply when $V \sim \lambda$: then the functions have no guaranteed curvature
 $$-\Delta \phi = (\lambda - V)\phi \sim 0$$ and the doubling in size, if present at all, could possible vanish once an affine function is removed.\\

\textbf{Acknowledgement.}  I am grateful to Ronald R. Coifman and Peter W. Jones for extensive discussions. The author was partially supported by an AMS-Simons travel grant
and INET grant $\#$INO15-00038.

\end{document}